\documentclass[11pt,reqno]{amsart}

\usepackage{enumerate}
\usepackage{color}
\newtheorem{theorem}{Theorem}[section]

\newtheorem{prop}{Proposition}[section]
\newtheorem{lemma}[theorem]{Lemma}
\newtheorem{defn}{Definition}[section]
\theoremstyle{remark}
\newtheorem{remark}{Remark}[section]

\numberwithin{equation}{section}
\DeclareMathOperator\supp{\rm supp}
 \DeclareMathOperator\hdim{\dim_H}
 \DeclareMathOperator\pdim{\dim_P}
 
\DeclareMathOperator*{\esssup}{ess\,sup}
\def\ue{\mathop{\rm \overline{\dim}_e}}
\def\lowe{\mathop{\rm \underline{\dim}_e}}
\def\fdim{\mathop{\rm \dim_F}}
\def\updim{\mathop{\rm \overline{\dim}_P}}
\def\N{\mathbb{N}}
\def\R{\mathbb{R}}
\def\Z{\mathbb{Z}}

 \usepackage{ulem}
\begin{document}

\title[On the intermediate value property]{On the intermediate value property of spectra for a class of Moran spectral measures}

\author[Jinjun Li]{Jinjun Li}

 \address[Jinjun Li]{School of Mathematics and Information Science, Guangzhou University, Guangzhou, 510006, P.~R.~China}
\email{li-jinjun@163.com}


\author{Zhiyi Wu}
\address[Zhiyi Wu]{School of Mathematics and Information Science, Guangzhou University, Guangzhou, 510006, P.~R.~China}
\email{zhiyiwu@126.com}




\begin{abstract}
We prove that the Beurling dimensions of the spectra for a class of Moran spectral measures are between $0$ and their upper entropy dimensions. Moreover, for such a Moran spectral measure $\mu$, we show that the Beurling dimension for the spectra of $\mu$ has the intermediate value property: let $t$ be any value between $0$ and the upper entropy dimension of $\mu$, then there exists a spectrum whose Beurling dimension is $t.$ In particular, this result settles affirmatively a conjecture involving spectral Bernoulli convolution proposed by Fu, He and Wen in [J. Math. Pures Appl. 116 (2018), 105--131]. Furthermore, we prove that the set of the spectra whose Beurling dimensions are equal to any fixed value between $0$ and $\ue \mu$ has the cardinality of the continuum.
\end{abstract}

\maketitle
\section{Introduction and main results}
Let $\mu$ be a Borel probability measure with compact support in $\R^d.$ We say that $\mu$ is a \textit{frame spectral measure} if there exist two constants $A, B>0$ and  a countable set $\Lambda\subset \R^d$, called a \textit{frame spectrum},  such that for every $f\in L^2(\mu)$, we have
\begin{equation}\label{frame}
A\|f\|^2\le \sum_{\lambda\in \Lambda}|\langle f, e_\lambda \rangle|^2\le B\|f\|^2,
\end{equation}
where $\langle \cdot, \cdot\rangle $ is the inner product in $L^2(\mu)$, $e_\lambda(x)=e^{2\pi i x\cdot\lambda}$ and $x\cdot\lambda=\sum_{i=1}^dx_i\lambda_i$ is the standard inner product in $\R^d$. Specially, we call $\mu$ a \textit{Parseval frame spectral measure} with \textit{Parseval frame spectrum} $\Lambda$ if $A=B=1$ in equation $\eqref{frame}$.  In this case, it is easy to see that the set $\{e_{\lambda}\}_{\lambda\in\Lambda}$ is an orthonormal basis for $L^2(\mu)$. Due to this fact, we also call $\mu$ a \textit{spectral measure} with \textit{spectrum} $\Lambda$. Many spectral measures have been found, see, for example \cite{AH,AHL,D,D16,DHL,DS,De,DC,DHLai19,F74,HL,LWY,L2,L1, LDL, LMW, JP} and references therein. For a singular spectral measure, generally speaking, its spectrum containing 0 is not unique, see \cite{ADH,D16,HTW, FHW18,FHW18,LWY, LWu, LWu2} and references therein. However, it is a big challenge to determine all the spectra for a given singular spectral measure. To the best of our knowledge, there is no one singular spectral measure whose spectra have been completely characterized. Therefore, it is interesting to investigate the structure of the spectra from possible aspects for a given singular spectral measure and find its spectra as many as possible. In this paper, we are interested in a class of singular spectral measures and show that the Beurling dimensions of their spectra possess an interesting intermediate value property. Moreover, we show that the set of the spectra whose Beurling dimensions are equal to any fixed value  has the cardinality of the continuum. These are the first results in this direction and our results show that this kind of spectral measures usually have spectra in abundance. Our results shed some light on the structure of the spectra for a singular spectral measure although we can not completely characterize them.

Let $\mathcal{B}=\{b_n\}_{n=1}^\infty$ and $\mathcal{D}=\{q_n\}_{n=1}^\infty$ be two sequences of integers with $2\le q_n<b_n$. For each $k\ge 1$, write $\mathcal{D}_k=\{0,1,\ldots, q_k-1\}$ and let
\begin{equation}\label{MM}
\mu=\mu_{\mathcal{D},\mathcal{B}}:=\delta_{b_1^{-1}\mathcal{D}_1}\ast \delta_{(b_1b_2)^{-1}\mathcal{D}_2}\ast\cdots \ast \delta_{(b_1b_2\cdots b_k)^{-1}\mathcal{D}_k}\ast\cdots,
\end{equation}
where $\delta_{\mathcal{E}}=\frac{1}{\#\mathcal{E}}\sum_{e\in \mathcal{E}}\delta_e$, $\#\mathcal{E}$ denotes the cardinality of the set $\mathcal{E}$, $\delta_e$ denotes the Dirac measure at $e$ and $\ast$ is the convolution operator.
We call $\mu$ a \textit{Moran measure} and it becomes the \textit{Cantor measure} $\mu_{q, b}$ studied in \cite{DHL} when $b_n=b$ and $q_n=q$ for all $n\ge 1.$ The support of $\mu$ can be represented in terms of digit expansion
\begin{equation}\label{supp}
\supp \mu=\left\{\sum_{k=1}^\infty \frac{d_k}{b_1\cdots b_k}:~d_k\in \mathcal{D}_{k}, k\ge 1\right\}.
\end{equation}
It can be regarded as a special Moran set introduced by Feng, Wen and Wu \cite{FWW}.

It is proved in \cite{AH}
that $\mu$ is a spectral measure if $q_n| b_n$ for all $n\ge 1$ and the authors also found a special spectrum for the Moran measure. In this paper, we always assume that the sequences $\mathcal{B}$ and $\mathcal{D}$ satisfy the following two technical conditions:
\begin{equation}\label{cond1}
2\le r_n:=\frac{b_n}{q_n}\in \N
\end{equation}
and
\begin{equation}\label{conn}
	M:=\sup_n \{b_n\}<\infty.
\end{equation}

For example, if $q_n=q, b_n=b$ for all $n\ge 1$ and $2\le b/q\in \N$, then both \eqref{cond1} and \eqref{conn} are satisfied.

The assumption \eqref{cond1} guarantees the singularity of $\mu_{\mathcal{D}, \mathcal{B}}$. There have been some results showing that there is a close relationship between the dimensions of the frame spectral measure and the Beurling dimensions of its spectra. For example, Dutkay et al. \cite{DHSW} showed that the Beurling dimension of any frame spectrum of the self-similar frame spectral measure with the same contracting ratios, equal probability weights and satisfying the open set condition is bounded by the Hausdorff dimension of its support. He et al. \cite{HKTW} proved that this result remains true if the assumptions that the contracting ratios are same and probability weights are equal are removed. In general, He et al. \cite{HKTW} proposed the following conjecture: if $\mu$ is a frame spectral measure with spectrum $\Lambda$ and compact support $T$ then $$\dim \Lambda\le \hdim T,$$ where $\dim \Lambda$ and $\dim_H T$ denote the Beurling dimension of $\Lambda$ and the Hausdorff dimension of $T$, respectively.  However, Shi \cite{Shi} constructed a spectral measure such that there exists a spectrum of it whose Beurling dimension is strictly greater than the Hausdorff dimension of its support. That is, for a  frame spectral measure, the Hausdorff dimension of its support is not the natural upper bound of the Beurling dimensions of its spectra.  In fact, for any frame spectral measure, Shi \cite{Shi} proved that the Beurling dimensions of its frame spectra are bounded by its upper entropy dimension. On the other hand, Theorem 1.3 in \cite{ILLW} implies that the Fourier dimension of a frame spectral measure is the lower bound of the Beurling dimensions of its frame spectra. More precisely, if $\mu$ is a frame spectral measure with frame spectrum $\Lambda$, then
\begin{equation}\label{kongzhi}
\fdim \mu\le \dim \Lambda\le \ue \mu,
 \end{equation}
where $\fdim \mu $ and $\ue \mu$ denote the Fourier dimension and upper entropy dimension of $\mu$, respectively. The definitions of these dimensions will be given in the next section.

For the Moran measure defined in \eqref{MM}, we will show that $\fdim \mu=0$ (Proposition \ref{lb}) and $\ue \mu$ can be expressed by a formula associated with $\mathcal{B}$ and $\mathcal{D}$ (Proposition \ref{ub}). Therefore, it follows from \eqref{kongzhi} that the Beurling dimensions for the spectra of $\mu$ are between $0$ and $\ue \mu$. It is well known that a spectral measure may admit various spectra but we usually don't know how ``dense" they are. For example, does there exist some spectrum which is sparse that whose Beurling dimension is zero? In fact,  Dai, He and Lai \cite{DHL} presented some spectra with zero Beurling dimension for some kind of Cantor measures $\mu_{q,b}$, and An and Lai \cite{AL} proved that there is spectrum with zero Beurling dimension for general self-affine spectral measure. These results are somewhat surprising since Lebesgue measure has no spectrum with Beurling dimension with zero due to the classical result of Landau \cite{Lan}.

It is natural to ask that whether there are some spectra whose Beurling dimensions are equal to  other values? In this paper, we show that the Beurling dimension for the spectra of $\mu$ has the following intermediate value property.

\begin{theorem}\label{MR1}
Let $\mu=\mu_{\mathcal{D},\mathcal{B}}$ be the Moran measure defined in \eqref{MM} satisfying \eqref{cond1} and \eqref{conn}. For any $t\in [0, \ue \mu]$, there exists a spectrum $\Lambda_t$ of $\mu$ such that $\dim \Lambda_t =t.$
\end{theorem}

\begin{remark}
We remark that condition \eqref{conn} guarantees that we can calculate the upper local dimension, and further the upper entropy dimension of $\mu$ by the formula in \cite{LW}, see Proposition \ref{ub}. When $q_n=q,b_n=b$ for all $n\geq1$, $\ue\mu$ coincides with $\hdim (\supp \mu)$ (see Proposition \ref{ub} in Section 2). In particular, let $q_n=2,b_n=2k$ for some $k\in\N$, then $\mu$ becomes the spectral Bernoulli convolution and Theorem \ref{MR1} settles affirmatively a conjecture proposed by Fu, He and Wen in \cite[Conjecture 5.3]{FHW18}.
\end{remark}

Furthermore, we can describe the cardinality of the spectra whose Beurling dimensions are equal to a fixed value between $0$ and $\ue \mu.$ More precisely, we have the following  surprising result.
\begin{theorem}\label{MR2}
 Let $\mu=\mu_{\mathcal{D},\mathcal{B}}$ be the Moran measure defined in \eqref{MM} satisfying \eqref{cond1} and \eqref{conn}.
  For any $t\in [0, \ue \mu]$, the level set
 \[
 L_t:=\{\Lambda:  \text{$\Lambda$ is a spectrum of $\mu$ and $\dim \Lambda =t$} \}
 \]
 has the cardinality of the continuum.
\end{theorem}

\section{Preliminaries}
This section is devoted to the definitions and some properties of Fourier dimension and upper entropy dimension for a measure, and the Beurling dimension of a countable set in $\R^d$.
\subsection{Fourier dimension}
The Fourier dimension of a measure is a measurement of exactly how fast the Fourier transformation of the measure decays and it can reflect some intrinsic  properties of the measure. More precisely, for a Borel probability measure on $\R^d$, its \textit{Fourier dimension} is defined as
\[
\begin{split}
\fdim \mu=\sup\Big\{0\le s\le d: ~ &\text{there exists a constant $C>0$}\\
 &\text{such that $|\widehat{\mu}(\xi)|\le C|\xi|^{-s/2}$ for all $\xi \in \mathbb{R}^d$}\Big\},
\end{split}
\]
where $\widehat{\mu}(\xi)=\int e^{-2\pi i\xi \cdot x}d\mu(x)$ is the Fourier transformation of $\mu$.

If $\mu$ is a frame spectral measure with frame spectrum $\Lambda$, then Theorem 1.3 in \cite{ILLW} implies that
\[
\dim \Lambda\ge \fdim \mu.
\]
That is, the Fourier dimension of a frame spectral measure is the lower bound of the Beurling dimension of its spectra.  We can obtain the following result by using almost the same arguments in \cite{CKL}. However, we present the complete proof here for the convenience of the reader.
\begin{prop}\label{lb}
Let $\mathcal{B}$ and $\mathcal{D}$ satisfy $r_n\ge 2$ for all $n\ge 1$ and let $\mu=\mu_{\mathcal{D},\mathcal{B}}$ be the Moran measure defined in \eqref{MM}. Then $\fdim \mu=0.$
\end{prop}
\begin{proof}
It is sufficient to show that the Fourier transformation of $\mu$ does not decay. We will prove it by employing the following classical method. For a set $E\subset [0,1]$ and an integer $n$, we define
\[
n E=\{nx(\bmod1):x\in E\}.
\]
If there is an increasing sequence $\{n_k\}$ of positive integers such that
\[
\overline{\bigcup_{k=1}^\infty n_k E}\not=[0,1],
 \]
where $\overline{A}$ denotes the closure of $A$ in the Euclidean topology. Then the Fourier transformation of any probability measure on $E$ does not decay. Its proof can be found in \cite{CKL} or \cite{Ly}.

For $k\ge 1$, let $n_k=b_1\cdots b_k$.  Then, by the representation of $\supp \mu$ in \eqref{supp}
\[
n_k(\supp \mu)=\left\{\sum_{j=1}^\infty \frac{d_{k+j}}{b_{k+1}\ldots b_{k+j}}:~ d_{k+j}\in \mathcal{D}_{k+j}, j\ge 1\right\}.
\]
Noting that $b_k\ge 4$ and $r_k\ge 2,$ we have
\[
\begin{split}
\max (n_k(\supp \mu))&=\sum_{j=1}^\infty \frac{ q_{k+j}-1}{b_{k+1}\ldots b_{k+j}}\le \sum_{j=1}^\infty \frac{1}{b_{k+1}\ldots b_{k+j-1}r_{k+j}}\\
&\le \frac{1}{2}\sum_{j=1}^\infty\frac{1}{4^{j-1}}<1.
\end{split}
\]
This implies that $\overline{n_k(\supp \mu)}\not=[0,1]$ and therefore $\fdim \mu=0.$
\end{proof}

\subsection{Entropy dimensions}
Entropy dimension for a probability measure is frequently used in fractals and dynamical systems.
Let $\mathcal{P}_n$ be the $n$-th dyadic partition of $\R^d$:
\[
\mathcal{P}_n=\{I_1\times I_2\times \cdots\times I_d:~I_j\in \mathcal{P}^{(1)}_n, 1\le j\le d\},
\]
where
\[
\mathcal{P}^{(1)}_n =\left\{\left[\frac{k}{2^n},\frac{k+1}{2^n}\right):k\in \Z\right\}.
\]
Write
\[
H_n(\mu)=-\sum_{Q\in \mathcal{P}_n}\mu(Q)\log \mu(Q).
\]

The \textit{upper entropy dimension} of $\mu$ is defined by
\[
\ue \mu=\limsup\limits_{n\to \infty}\frac{H_n(\mu)}{\log 2^n}.
\]
The \textit{lower entropy dimension} $\lowe \mu$ is defined similarly by taking the lower limit.
\begin{prop}\label{ub}
Let $\mu=\mu_{\mathcal{D},\mathcal{B}}$ be the Moran measure defined in \eqref{MM} and \eqref{conn} holds. Then
\begin{equation}\label{su}
\ue \mu=\limsup\limits_{n\to \infty}\frac{\log q_1q_2\cdots q_n}{\log b_1b_2\cdots b_n}.
\end{equation}
\end{prop}
Let's remark that the Hausdorff dimension of the support of $\mu$ is generally less than $\ue \mu$. In fact the set $\supp \mu$ can be regarded as a partial homogeneous Moran set and Lemma 2.2 in \cite{FWW} shows that its Hausdorff dimension is
\[
\hdim \supp \mu=\liminf\limits_{n\to \infty}\frac{\log q_1q_2\cdots q_n}{\log b_1b_2\cdots b_n+\log (b_{n+1}/q_{n+1})},
\]
and then the condition \eqref{conn} implies that
\[
\hdim \supp \mu=\liminf\limits_{n\to \infty}\frac{\log q_1q_2\cdots q_n}{\log b_1b_2\cdots b_n}.
\]
When $q_n=q,b_n=b$ for all $n\geq1$, then Proposition \ref{ub} implies that $\ue \mu=\hdim\supp\mu$.

We will show that there is a spectrum $\Lambda$ of $\mu$ such that $\dim \Lambda= \ue \mu$ in the proof of Theorem \ref{MR1}. On the other hand, it is easy to construct sequences $\{b_n\}_{n=1}^\infty$ and $\{q_n\}_{n=1}^\infty$ such that
\[
\dim \Lambda=\limsup\limits_{n\to \infty}\frac{\log q_1q_2\cdots q_n}{\log b_1b_2\cdots b_n}>\liminf\limits_{n\to \infty}\frac{\log q_1q_2\cdots q_n}{\log b_1b_2\cdots b_n}=\hdim \supp \mu.
\]
This provides counterexamples to the conjecture proposed by He et al. \cite{HKTW}.

To prove Proposition \ref{ub}, we need some notation. For a Borel probability measure $\mu$ in $\R^d$, its upper packing dimension is defined by
\[
\updim \mu=\inf\{\pdim F:~ \mu(F)=1\},
\]
where $\pdim F$ is the packing dimension of $F$. For its definition and properties, the reader is referred to the famous books \cite{Fal1,Fal}. The upper packing dimension of a measure is closely related to its upper local dimension, which is defined by
 \[
 \overline{d}(\mu,x)=\limsup\limits_{r\to 0}\frac{\log \mu(B(x,r))}{\log r}, ~~x\in \R^d.
 \]
 It is well known that $\updim \mu$ is equal to the essential supremum of $\overline{d}(\mu,x)$. That is,
 \begin{equation}\label{upd}
 \updim \mu=\esssup \overline{d}(\mu,x),
 \end{equation}
see, for example, \cite{Edg} or \cite{Fal}.

There is an equivalent definition for upper entropy dimension.
\begin{lemma}[Theorem 2.1 in \cite{BGT}]\label{eq}
Let $\mu$ be a Borel probability measure on $\R^d.$ Then
\[
\ue \mu=\limsup\limits_{r\to 0}\frac{\int_{\supp \mu}\log \mu(B(x,r))d\mu(x)}{\log r}.
\]
\end{lemma}
Now we turn to the proof of Proposition \ref{ub}.

\begin{proof}[Proof of Proposition \ref{ub}]
It is well known  that $\ue \mu\le \updim \mu$, see, for example,  Theorem 1.3 in \cite{FLR}. On the other hand, the formula provided in \cite{LW} shows that
\[
 \overline{d}(\mu,x)=\limsup\limits_{n\to \infty}\frac{\log q_1q_2\cdots q_n}{\log b_1b_2\cdots b_n}
\]
for $\mu$ almost every $x\in\supp \mu.$ In fact, the upper local dimension formula for a more general class of measures is presented in \cite{LW}. Therefore, by \eqref{upd} we get the upper bound in $\eqref{su}.$

Next we use Lemma \ref{eq} to get the lower bound.  For each $n\ge 1,$ define
\[
\mathcal{A}_n =\left\{\left[\frac{k}{b_1\cdots b_n},\frac{k+1}{b_1\cdots b_n}\right):k\in \Z\right\}.
\]

Clearly, each $\mathcal{A}_n$ is a partition of $\R.$
For any $x\in \supp \mu$ and $r_n'=\frac{1}{b_1\cdots b_n}$, we have
\begin{equation}\label{contain}
B(x, r_n')\subset \frac{3}{2} I,
\end{equation}
where $x\in I\in \mathcal{A}_n$ and $\frac{3}{2} I=B(x_I,\frac{3}{2}|I|)$. Here, $x_I$ is the centre of $I$ and $|I|$ denotes the Lebesuge measure of $I$. The conclusion \eqref{contain} implies that
\[
\begin{split}
\int_{\supp \mu}\log\mu(B(x,r_n'))d\mu(x)&\le \sum_{I\in \mathcal{A}_n}\mu(I)\log (3\mu(I))\\
                                        &=\log 3 -\log q_1\cdots q_n.
\end{split}
\]
Therefore, it follows from Lemma \ref{eq} that
\[
\begin{split}
\ue \mu&\ge \limsup\limits_{n\to \infty}\frac{\int_{\supp \mu}\log\mu(B(x,r_n'))d\mu(x)}{\log r_n'}\\
       &\ge \limsup\limits_{n\to \infty}\frac{\log q_1q_2\cdots q_n}{\log b_1b_2\cdots b_n}.
\end{split}
\]
This proves the lower bound in $\eqref{su}$ and  therefore the proof of Proposition \ref{ub} is completed.
\end{proof}

\subsection{Beurling dimension}
Czaja, Kutyniok and Speegle in \cite{CKS} introduced the notion of Beurling dimension for discrete subsets of $\R^d$ via Beurling density and employed it to study Gabor pseudoframes. From then on, Beurling dimension has wide applications in the study of (frame) spectra for (frame) spectral measures, see \cite{DHL,DHSW,HKTW,Shi} and references therein. In this paper, we only consider the discrete set in $\R.$
Let $\Lambda\subset \R$ be a countable set. For $r>0,$ the \textit{upper $r$-Beurling density} of $\Lambda$ is defined by
\[
D_r^+(\Lambda)=\limsup\limits_{h\to \infty}\sup_{x\in \R}\frac{\#(\Lambda \cap B(x,h))}{h^r}.
\]
It is not difficult to prove that there is a critical value of $r$ where $D_r^+(\Lambda)$ jumps from $\infty$ to $0$. This critical value is defined as the \textit{Beurling dimension} of $\Lambda.$ More precisely
\[
\dim \Lambda=\inf\{r:D_r^+(\Lambda)=0\}=\sup\{r:D_r^+(\Lambda)=\infty\}.
\]

We remark that the Beurling dimension here is called upper Beurling dimension in \cite{CKS} and we can similarly define the lower Beurling dimension, which is not needed in this paper.
The following proposition, established in \cite{CKS}, provides a way for calculating the Beurling dimension of a countable set.
\begin{prop}[\cite{CKS}]
\label{ec}
Let $\Lambda\subset \R$ be a countable set. Then
\begin{equation}\label{jsf}
\dim \Lambda=\limsup\limits_{h\to \infty}\sup_{x\in \R^d}\frac{\log \#(\Lambda \cap B(x,h))}{\log h}.
\end{equation}
\end{prop}
Clearly,  formula \eqref{jsf} is equivalent to the following
\begin{equation}\label{jsf2}
\dim \Lambda=\limsup\limits_{|I|\to \infty}\frac{\log \#(\Lambda \cap I)}{\log |I|},
\end{equation}
where $I$ is taken over all intervals in $\R$ with length $|I|\to \infty.$

\begin{lemma}\cite{CKS}\label{inc}
	Let $\Lambda_1,\Lambda_2\subset \R$ be two countable sets. Then the following statements hold:
\begin{enumerate}
\item [\textup{(i)}] Monotonicity: If $\Lambda_1\subset \Lambda_2$, then
		\[\dim\Lambda_1\leq \dim\Lambda_2;\]
\item [\textup{(ii)}] Finite stability:
		\[\dim(\Lambda_1\cup \Lambda_2)=\max(\dim\Lambda_1,\dim\Lambda_2).\]
\end{enumerate}
	
\end{lemma}

\begin{defn}
	A countable set $\Lambda=\{a_n\}_{n=0}^{\infty}\subset \R
	$ is called $b$-lacunary, if $a_0 = 0$,
	$|a_1|\geq b$ and for all $n \geq 1$,
	$|a_{n+1}| \geq b|a_n|$.
\end{defn}
\begin{lemma}\cite{AL}\label{lac}
	Let $b>1$. If $\Lambda$ is a $b$-lacunary set, then $\dim \Lambda=0$.
\end{lemma}

Generally speaking, it is difficult to calculate the Beurling dimension of a countable set unless it has particular structure.  We next provide a formula of Beurling dimension for a class of special discrete sets in $\R$, which plays a key role in the proof of Theorem \ref{MR1}. Let $\{t_k\}$ be a sequence of positive integers. Let $\{n_k\}_{k=0}^\infty, \{m_k\}_{k=1}^\infty$ be two sequences of positive integers satisfying
$n_0=1$ and
\begin{equation}\label{cfd}
\text {$t_{k+1}n_1\cdots n_k>\sum_{i=1}^{k}(m_i-1)t_in_1\cdots n_{i-1}$ \quad for any $k\ge 1.$ }
\end{equation}
For example, if $t_k=n_k/m_k$ for $k\geq1$, \eqref{cfd} holds. In fact, for any $k\geq1$, \[
\begin{split}
&t_{k+1}n_1\cdots n_k-\sum_{i=1}^{k}(m_i-1)t_in_1\cdots n_{i-1}\\
\geq& t_{k+1}n_1\cdots n_k+t_1+t_2n_1+\cdots+t_kn_1\cdots n_{k-1}\\
& -n_1-n_1n_2-\cdots-n_1\cdots n_k\\
>&0
\end{split}
\]
by the fact $t_k=n_k/m_k\geq2$.

For any $k\ge 1,$ write
\[
\mathcal{M}_k=\left\{\sum_{i=1}^k x_it_i(n_1\cdots n_{i-1}):~\text{$x_i\in \{0,1,\ldots, m_i-1\}$ for $1\le i\le k$}\right\}
\]
and define
\begin{equation}\label{IMS}
\mathcal{M}:=\mathcal{M}(\{n_k\},\{m_k\}, \{t_k\})=\bigcup_{k=1}^\infty \mathcal{M}_k.
\end{equation}
We call $\mathcal{M}$ an \textit{integer Moran set} associated with the data $\{n_k\}$, $\{m_k\}$ and $\{t_k\}$. The definition of integer Moran set is motivated by the structures of some spectra of known spectral measures.

%

  Since integer Moran set has traceable construction, its Beurling dimension can be calculated. More precisely, we have the following result.
\begin{prop}\label{Bdims}
Let $\mathcal{M}$ be the integer Moran set defined as in \eqref{IMS} and the associated date $\{m_k\}, \{n_k\}$ and $\{t_k\}$ satisfy the assumption \eqref{cfd}. Then, its Beurling dimension is given by the following formula
\begin{equation}\label{Bf}
\dim \mathcal{M}=\limsup\limits_{k\to \infty}\frac{\log m_1\cdots m_k}{\log m_kt_k n_1\cdots n_{k-1}}.
\end{equation}
\end{prop}

\begin{proof}
First, if there exists $K\ge 1$ such that $m_k=1$ for all  $k\ge K$, then the corresponding set $\mathcal{M}$ is a finite set and therefore $\dim \mathcal{M}=0.$ That is, \eqref{Bf} holds. Hence, we next assume that there are finitely many $k$ such that $m_k\ge 2.$

For positive integer $k$ with $m_k\ge 2$, let $J_k=[0, \sum_{i=1}^k (m_i-1)t_i(n_1\cdots n_{i-1})]$. Then $\#(\mathcal{M}\cap J_k)=\#\mathcal{M}_k=m_1\cdots m_k$ and we can use $\eqref{cfd}$ to estimate the length of $J_k$:
\[
\begin{split}
|J_k|&=m_kt_k(n_1\cdots n_{k-1})-\left(t_kn_1\cdots n_{k-1}-\sum_{i=1}^{k-1}(m_i-1)t_in_1\cdots n_{i-1}\right)\\
   &\le m_kt_k n_1\cdots n_{k-1}.
\end{split}
\]
Hence, by Proposition \ref{ec} we have
\[
\begin{split}
\dim \mathcal{M}&=\limsup\limits_{|I|\to \infty}\frac{\log \#(\mathcal{M} \cap I)}{\log |I|}\quad\quad (\text{by $\eqref{jsf2}$})\\
            &\ge \limsup\limits_{k\to \infty}\frac{\log \#(\mathcal{M} \cap J_k)}{\log |J_k|}\\
            &\geq\limsup\limits_{k\to \infty}\frac{\log m_1\cdots m_k}{\log m_k t_kn_1\cdots n_{k-1}}.
\end{split}
\]

We next prove that
\begin{equation}\label{upperb}
\dim \mathcal{M}\le \limsup\limits_{k\to \infty}\frac{\log m_1\cdots m_k}{\log m_kt_k n_1\cdots n_{k-1}}.
\end{equation}

 Assume that $I=[a,b]$ is an interval with $a, b\in \N, a<b$.  We can assume that $a,b\in \mathcal{M}$. In fact, if not, we can slightly shorten the interval $I$ to $I'$ such that the endpoints of $I'$ are in $\mathcal{M}$ and
 $\#(\mathcal{M}\cap I')= \#(\mathcal{M}\cap I).$
 To prove \eqref{upperb}, it is sufficient to consider the interval $I'$ since
 \[
  \limsup\limits_{|I|\to \infty}\frac{\log \#(\mathcal{M}\cap I)}{\log |I|}\le  \limsup\limits_{|I'|\to \infty}\frac{\log \#(\mathcal{M}\cap I')}{\log |I'|}.
 \]
Therefore, there exist some $k\ge 1$ and $x_i', x_i\in\{0,1,\ldots, m_i-1\}~(1\le i\le k),  x_k\not=0$ such that
\[
a=\sum_{i=1}^k x_i't_i(n_1\cdots n_{i-1}),\quad b=\sum_{i=1}^k x_it_i(n_1\cdots n_{i-1}).
\]

Clearly, $x_k\ge x_k'$. We can assume that $x_k'\not= x_k.$ Otherwise, we consider the interval $I-x_kt_kn_1\cdots n_{k-1}$. We divide the proof into two cases.

\textsc{Case 1.} $x_k-x_k'\ge 2.$ Clearly, in this case, all $m_k>2$. We can get the following estimates:
\[
\#(\mathcal{M} \cap [a,b])\le m_1\cdots m_{k-1}(x_k-x_k')
\]
and
\[
\begin{split}
b-a&= (x_1-x_1')t_1+(x_2-x_2')t_2n_1+\cdots+(x_k-x_k')t_k(n_1\cdots n_{k-1})\\
             &=(x_1-x_1')t_1+(x_2-x_2')t_2n_1+\cdots+(x_{k-1}-x_{k-1}')t_{k-1}(n_1\cdots n_{k-2})\\
             &\phantom{\le}+(n_1\cdots n_{k-1})t_k+(x_k-x_k'-1)t_k(n_1\cdots n_{k-1})\\
             &\ge (x_k-x_k'-1)t_k(n_1\cdots n_{k-1}),
\end{split}
\]
where the last inequality follows from the assumption $\eqref{cfd}$.

Let
\[
f(s)=\frac{\log s(m_1\cdots m_{k-1})}{\log (s-1)t_k(n_1\cdots n_{k-1})}:=\frac{\log sa_{k-1}}{\log (s-1)c_{k-1}},\quad s\geq2.
\]

Note that \[
\begin{split}
f(s+1)-f(s)&=\frac{\log (s+1)a_{k-1}}{\log sc_{k-1}}-\frac{\log sa_{k-1}}{\log (s-1)c_{k-1}}\\
&=\frac{\log (s+1)a_{k-1}\cdot\log (s-1)c_{k-1}-\log sa_{k-1}\cdot\log sc_{k-1}}{\log sc_{k-1}\cdot\log (s-1)c_{k-1}}.
\end{split}
\]
Because the denominator of the above formula is greater than $0$, we only consider the  numerator.
Write
\[
\begin{split}
g(s)=\log (s+1)a_{k-1}\cdot\log (s-1)c_{k-1}-\log sa_{k-1}\cdot\log sc_{k-1}.
\end{split}
\]
By elementary calculation and the assumption $n_k\ge 2m_k, m_k\ge 2, t_k\ge 2$ for all $k\ge 1$, we have
\[
\begin{split}
g(s)&=\log (s+1)\log (s-1)-\log s\log s+\log c_{k-1}\log \left(\frac{s+1}{s}\right)\\
&\phantom{\le}+\log a_{k-1}\log \left(\frac{s-1}{s}\right)\\
&\ge \log (s+1)\log (s-1)-\log s\log s+\log (2^{k-1}t_ka_{k-1})\log \left(\frac{s+1}{s}\right) \\
&\phantom{\le}+\left(\log \frac{s-1}{s}\right)\log a_{k-1}\\
&\ge \log (s+1)\log (s-1)-\log s\log s\\
&\phantom{\le}+\left((2k-1)\log \frac{s+1}{s}+(k-1)\log \frac{s-1}{s}\right)\log 2\\
&>0
\end{split}
\]
for sufficiently large $k$. This implies that $f(s)$ (for sufficiently large $k$) is increasing with respect to $s$, and noting that $x_k-x_k'\le m_k-1$ we have

\[
\begin{split}
\limsup\limits_{|I|\to \infty}\frac{\log \#(\mathcal{M} \cap I)}{\log |I|}&\le \limsup\limits_{k\to \infty}\frac{\log (x_k-x_k')(m_1\cdots m_{k-1})}{\log (x_k-x_k'-1)t_k(n_1\cdots n_{k-1})}\\
&\le \limsup\limits_{k\to \infty}\frac{\log (m_k-1)(m_1\cdots m_{k-1})}{\log (m_k-2)t_k(n_1\cdots n_{k-1})}\\
&\le\limsup\limits_{k\to \infty}\frac{\log m_1\cdots m_k}{\log m_kt_k(n_1\cdots n_{k-1})(1-\frac{2}{m_k})}\\
&=\limsup\limits_{k\to \infty}\frac{\log m_1\cdots m_k}{\log m_kt_kn_1\cdots n_{k-1}}.
\end{split}
\]
This, combined with \eqref{jsf2}, implies that \eqref{upperb} holds.

\textsc{Case 2.}  $x_k=x_k'+1$.  Let
\[
c=\sum_{i=1}^{k-1} (m_i-1)t_i(n_1\cdots n_{i-1})+x_k't_kn_1\cdots n_{k-1}.
\]
It follows from the assumption \eqref{cfd} that $c<b.$ We split $[a,b]$ into two disjoint subintervals $[a,b]=[a,c]\cup (c,b]$. Without loss of generality, we suppose $\#(\mathcal{M} \cap (c, b])\leq\#(\mathcal{M} \cap [a, c])$. Then  we have
\[
\#(\mathcal{M} \cap [a, b])\leq2\#(\mathcal{M} \cap [a, c]).
\]
If $m_i-1-x_i'\le 1$ for all $1\le i\le k-1$, we claim $\limsup\limits_{|I|\to \infty}\frac{\log \#(\mathcal{M} \cap I)}{\log |I|}=0$ and therefore $\eqref{upperb}$ holds trivially. In fact, if $x'_i=m_i-1$ for $1\le i\le k-1$ then $a=c$ and thus $\#(\mathcal{M} \cap [a, b])\le 2;$ if $x'_i=m_i-2$ for $1\le i\le k-1$ then we can check that $\#(\mathcal{M} \cap [a, c])\le k-1$ and $|I|\ge c-a=t_1+t_2n_1+\cdots+t_{k-1}n_1n_2\cdots n_{k-2}$. Noting that $n_k\ge 4$, we have
\[
\limsup\limits_{|I|\to \infty}\frac{\log \#(\mathcal{M} \cap I)}{\log |I|}\le \limsup\limits_{k\to \infty}\frac{\log 2(k-1)}{\log t_{k-1}n_1n_2\cdots n_{k-2}}=0.
\]

So, we next only need to consider the case that $m_i-1-x_i'\ge 2$ for some $1\le i\le k-1$.  Let
\[
u_k=\max\{1\le i\le k-1: m_i-1-x_i'\ge 2\},
\]
then $\#(\mathcal{M} \cap [a, c])\le m_1\cdots m_{u_k-1}(m_{u_k}-x_{u_k}')$.

On the other hand, by \eqref{cfd} again we have
\[
\begin{split}
b-a&\ge c-a\ge (m_1-1-x_1')t_1+\cdots+(m_{u_k}-1-x_{u_k}')t_{u_k}n_1\cdots n_{u_k-1}\\
    &\ge (m_{u_k}-x_{u_k}'-2)t_{u_k}n_1\cdots n_{u_k-1}.
\end{split}
\]

Therefore,  using the same argument in Case 1,
\[
\begin{split}
\limsup\limits_{|I|\to \infty}\frac{\log \#(\mathcal{M} \cap I)}{\log |I|}&\le \limsup\limits_{k\to \infty}\frac{\log (m_{u_k}-x_{u_k}')m_1\cdots m_{u_k-1}}{\log (m_{u_k}-x_{u_k}'-2)t_{u_k}n_1\cdots n_{u_k-1}}\\
& \le \limsup\limits_{k\to \infty}\frac{ \log m_1\ldots m_{u_k}}{\log m_{u_k}t_{u_k}n_1\cdots n_{u_k-1}}\\
&\le \limsup\limits_{k\to \infty}\frac{\log m_1\cdots m_k}{\log m_kt_kn_1\cdots n_{k-1}}.
\end{split}
\]


The proof of the lemma is completed.
\end{proof}

\section{Proofs of Theorems \ref{MR1} and \ref{MR2}}
This section is devoted to the proofs of Theorems $\ref{MR1}$ and \ref{MR2}.

For a singular spectral measure, it is challenging to characterize all its spectra.  However, it is possible to find some sufficient conditions to guarantee a discrete set is a spectrum for some spectral measure. In particular,  Dai and Sun \cite{DS} gave such a sufficient condition for the Moran measure $\mu$ by employing the tree labeling method initiated by Dutkay, Han and Sun \cite{DHS}. To state the result, we need some notation.  For $q\ge 1,$ write $\Sigma_q=\{0,1,\ldots, q-1\}$. For a sequence $\mathcal{D}=\{q_n\}_{n=1}^\infty$ of positive integers and $n\ge 1$, let
\[
\Sigma^0_\mathcal{D}=\vartheta, \quad  \Sigma^n_\mathcal{D}=\Sigma_{q_1}\times \Sigma_{q_2}\times \cdots \times \Sigma_{q_n},
\]
 and $\Sigma^*_\mathcal{D}=\cup_{n=1}^\infty \Sigma^n_\mathcal{D}.$

 For $m,n\ge 1$ and two words $\delta\in \Sigma^n_\mathcal{D}, \delta'\in \Sigma_{q_{n+1}}\times \cdots \times \Sigma_{q_{n+m}}$, we use $\delta \delta'$ to denote the concatenation of them. Of course, we set $\vartheta \delta=\delta$ for any word $\delta.$ We say $\Sigma^*_\mathcal{D}$ is a $\mathcal{D}$-\textit{adic tree} if we naturally set the root is $\vartheta,$ all the $n$-th level nodes are $\Sigma^n_\mathcal{D}$ and all the offsprings of $\delta\in \Sigma^n_\mathcal{D}$ are $\{\delta i:~i\in \Sigma_{q_{n+1}}\}$ for $n\ge 1$.

We say that $\tau: \Sigma^*_\mathcal{D}\to \R$ is a \textit{maximal tree mapping} if
\begin{enumerate}[(i)]
\item $\tau (\vartheta)=\tau({R_n(0^\infty)})=0$ for all $n\ge 1;$
\item $\tau(\delta_1\cdots \delta_n)\in (\delta_n+q_n \Z)\cap \{-1,0, 1, \ldots, b_n-2\}$ for $\delta_1\cdots \delta_n\in \Sigma^n_\mathcal{D}$ for $n\ge 1$;
\item for any word $\delta\in \Sigma^n_\mathcal{D}$ there exists $\delta'\in \Sigma_{q_{n+1}}\times \cdots \times \Sigma_{q_{n+m}}$ of length $m\ge 1$ such that $\tau(R_k(\delta\delta'0^\infty))=0$ for sufficiently large $k,$
\end{enumerate}
where $0^\infty:=000\cdots$ and $R_k(\delta)=\delta_1\cdots \delta_k\in \Sigma^k_\mathcal{D}$ for any $\delta\in \Pi_{n=1}^\infty \Sigma_{q_n}$.

The set $\{-1,0, 1, \ldots, b_n-2\}$ in $(\mathrm{ii})$ of the above definition can be replaced by any complete residual system modulo $b_n$ containing $0$. For example, Dai and Sun in their original paper \cite{DS} chose the set $\{-\lfloor b_n/2\rfloor,-\lfloor b_n/2\rfloor+1,\ldots,b_n-1-\lfloor b_n/2\rfloor\}$. In this paper, for our need, we choose the above set, which does not affect the following results.

For a maximal tree mapping $\tau,$ we are interested in the following set
\begin{equation}\label{ms}
\begin{split}
\Lambda(\tau)=\Bigg\{\sum_{n=1}^\infty \tau(R_n(\delta 0^\infty))\rho_n: ~\text{$\delta\in \Sigma^*_\mathcal{D}$ }&\text{with  $\tau(R_n(\delta 0^\infty))=0$} \\
&\text{ for sufficiently large $n$}\Bigg\},
 \end{split}
\end{equation}
where $\rho_n=r_nb_1\cdots b_{n-1}, ~ n\ge 1.$

The following criterion for spectra of $\mu$ is due to Dai and Sun \cite{DS}.

\begin{theorem}[\cite{DS}]
\label{cri}
Let $\tau$ be a maximal tree mapping and let $\Lambda=\Lambda(\tau)$ be defined as in \eqref{ms}. If
\[
\sup_{n\ge 1}\sup_{\delta\in \Sigma^n_{\mathcal{D}}}\#\left\{j\ge 1:~ \tau(R_{n+j}(\delta 0^\infty))\not=0\right\}<\infty,
\]
then $\Lambda(\tau)$ is a spectrum of the Moran measure $\mu$ defined in \eqref{MM}.
\end{theorem}
In fact, a stronger result (Theorem \ref{cri} is just a corollary) was presented in \cite{DS}. However, the above result is enough for us.

Now we are ready to prove Theorem \ref{MR1}. The difficult part is the case $0<t<\ue \mu$. Our main idea is to generalize the way of constructing a new spectrum for a self-similar spectral measure with consecutive digit sets in \cite{DHL} to our Moran case and divide the constructed spectra into two parts: regular part (subset of canonical spectra) and irregular part. We then  obtain the intermediate value property of Beurling dimension of regular part. Moreover, we can calculate that the Beurling dimension of irregular part is zero. By the finite stability of Beurling dimension, we obtain the desired result.

\begin{proof}[Proof of Theorem  \ref{MR1}]

For any $n\ge1,$  there exists a unique $k:=k_n$ and $\sigma=\sigma_1\cdots \sigma_{k}\in \Sigma^k_\mathcal{D}$ with $\sigma_k\not=0$ such that
\begin{equation}\label{exn}
	n=\sigma_1+\sigma_2q_1+\cdots+\sigma_{k}q_1\cdots q_{k-1}=\sum_{j=1}^{k}\sigma_jq_1\cdots q_{j-1}.
\end{equation}
In fact, for $n\ge 1$, there exist integers $t_1\ge 0$ and $\sigma_1\in \Sigma_{q_1}$ such that $n=q_1t_1+\sigma_1$. Similarly, we can write $t_1=q_2t_2+\sigma_2$, where $0\le t_2\in \N$ and $\sigma_2\in \Sigma_{q_2}$ and therefore $n=t_2q_1q_2+\sigma_2q_1+\sigma_1$. Continuing the procedure, there will exist some integer $k$ such that  $n=\sigma_1+\sigma_2q_1+\cdots+t_{k-1}q_1\cdots q_{k-1}, t_{k-1}\not=0$ and $t_{k-1}<q_{k}$. We stop at step $k-1$ and obtain \eqref{exn} by letting $t_{k-1}=\sigma_k$.

By the above construction, each positive integer  $n\ge 1$ corresponds a unique word $\sigma=\sigma(n):=\sigma_1\cdots \sigma_k\in \Sigma^k_\mathcal{D}$ with $\sigma_k\not=0$. Conversely, every word $\sigma=\sigma_1\cdots \sigma_k\in \Sigma^k_\mathcal{D}$ with $\sigma_k\not=0$, corresponds a unique integer $n:=n(\sigma)$ which has the expansion \eqref{exn}.

Let $\{s_n\}_{n=1}^\infty$ be a sequence of non-negative integers. We define a mapping by $\tau(\vartheta)=\tau(0^k)=0$ for $k\ge 1$, and for $n$ as in \eqref{exn}, $\tau(\sigma)=\sigma_k$, and
\begin{equation}\label{mm}
	\tau(\sigma0^l)=
	\begin{cases}
		0,& \text{if $l\not=s_n;$}\\
		q_{k+s_n},& \text{if $l=s_n.$}
	\end{cases}
\end{equation}
Define $\lambda_0=0$,  and
\[
\lambda_n=\sum_{j=1}^{k} \tau(\sigma_1\cdots \sigma_j)(r_jb_1\cdots b_{j-1})+b_1\cdots b_{k+s_n}\delta_{s_n},
\]
	\begin{equation*}
	\delta_{s_n}=
	\begin{cases}
		0,& \text{if $ ~s_n=0;$}\\
		1,& \text{if $ ~s_n\geq 1.$}
	\end{cases}
\end{equation*}
Then, it is easy to see that  $\tau$ is a maximal tree mapping and
\[
\text{$\ell_n:=\#\{k:\tau(\sigma0^l)\neq0,l\geq1\}\leq 1$ for all $n\geq1$.}
\]
By Theorem \ref{cri}, we have that $\Lambda(\{s_n\}):=\{\lambda_n\}_{n=0}^\infty$ is a spectrum of $\mu$ and we call $\Lambda(\{s_n\})$ the spectrum of $\mu$ with respective to the sequence $\{s_n\}$.

When $s_n=0$ for all $n\geq1$, it is easy to see that
\[
\Lambda(0):=\Lambda(\{s_n\})=\bigcup_{k=1}^\infty \left\{\sum_{i=1}^k x_i(r_ib_1\cdots b_{i-1}):~\text{$x_i\in \Sigma_{q_i}$ for $1\le i\le k$}\right\}.
\]
$\Lambda(0)$ is called the {\it canonical spectrum} of $\mu$. It follows from Proposition \ref{Bdims}, Proposition \ref{ub} and the fact that $r_kq_k=b_k$ for $k\ge 1$ that

\[
\begin{split}
	\dim \Lambda(0)&=\limsup\limits_{k\to \infty}\frac{\log q_1\cdots q_k}{\log q_k (r_kb_1\cdots b_{k-1})}\\
	&=\limsup\limits_{k\to \infty}\frac{\log q_1q_2\cdots q_k}{\log b_1b_2\cdots b_k}=\ue \mu.
\end{split}
\]

So, when $t=\ue \mu$ we choose $\Lambda_t=\Lambda(0)$.

We next consider the other two cases: $t=0$ and $0<t<\ue \mu.$

\textsc{Case 1}. $t=0$.  Consider the spectrum of $\mu$ with respective to $\{n\}$, i.e., $\Lambda(\{n\})=\{\lambda_n\}_{n=0}^{\infty}$. Noting that for $n\geq1$, we have $\lambda_n\geq b_1b_2\cdots b_{k_n+n}$ and
\[
\begin{split}
		\lambda_n&\leq b_1+b_1b_2+\cdots+ b_1b_2\cdots b_{k_n}+b_1b_2\cdots b_{k_n+n}\\
		&\leq2b_1b_2\cdots b_{k_n+n}.
\end{split}
\]
Noting that $b_n\ge 4$ for any $n\ge 1$, this implies that
\[
\frac{\lambda_{n+1}}{\lambda_n}\geq\frac{b_1b_2\cdots b_{k_{n+1}+n+1}}{2b_1b_2\cdots b_{k_n+n}}\geq \frac{b_{k_{n+1}+n+1}}{2}\ge 2.
\]
By Lemma \ref{lac}, we have that $\dim(\Lambda(\{n\}))=0$. So, take $\Lambda_0=\Lambda(\{n\}).$

\textsc{Case 2}. $0<t<\ue \mu$. We start from the following lemma.

\begin{lemma}\label{mid}
Suppose  $\{q_k\}_{k=1}^\infty, \{b_k\}_{k=1}^\infty$ be two sequences of positive integers satisfying \eqref{cond1} and \eqref{conn}. Then,  for any
\[
0<t <\limsup\limits_{k\to \infty}\frac{\log (q_1\cdots q_k)}{\log(b_1\cdots b_k)},
\]
there exists a sequence $q_k'\in \{1, 2,\ldots q_k\}, k\ge 1$ such that
\begin{equation}\label{ssp}
t=\limsup\limits_{k\to \infty}\frac{\log (q_1'\cdots q_k')}{\log (b_1\cdots b_{k})}.
\end{equation}
\end{lemma}

\begin{proof}
Let
\begin{equation}\label{assu}
0<t <\limsup\limits_{k\to \infty}\frac{\log (q_1\cdots q_k)}{\log(b_1\cdots b_k)}.
\end{equation}
We next inductively define a sequence of positive integers $\{k_j\}_{j=1}^\infty$  as follows. The idea is to replace some $q_k$s with $1$s at proper positions. Define $k_1=1$ if $t<\frac{\log q_2}{\log(b_1b_2)}$; otherwise, due to $\eqref{assu}$ we define
\[
k_1=\min\left\{i:\frac{\log (1\cdot q_2\cdots q_i)}{\log(b_1\cdots b_i)}\le t<\frac{\log (1\cdot q_2\cdots q_{i+1})}{\log(b_1\cdots b_{i+1})}\right\}
\]
and $q_1'=1, q_i'=q_i, 2\le i\le k_1+1.$

Then, define
\[
k_2=\min\left\{i:\frac{\log (q_1'\cdot \cdots q_{k_1+1}'1\cdots 1\cdot q_{i})}{\log(b_1\cdots b_i)}\le t<\frac{\log (q_1'\cdot\cdots q_{k_1+1}'1\cdots q_i\cdot q_{i+1})}{\log(b_1\cdots b_{i+1})}\right\}
\]
and $q'_{i}=1, k_1+2\le i\le k_2-1, q'_{k_2}=q_{k_2}, q'_{k_2+1}=q_{k_2+1}.$ Continuing inductively, we get a sequence of integers $\{k_j\}_{j=1}^\infty$ and a sequence of $\{q_k'\}_{k=1}^\infty$ such that
\begin{equation}\label{tps}
k_{j+1}=\min\left\{i:\frac{\log (q_1'\cdot \cdots q_{k_{j}+1}'1\cdots 1\cdot q_{i})}{\log(b_1\cdots b_i)}\le t<\frac{\log (q_1'\cdot\cdots q_{k_{j}+1}'1\cdots q_i\cdot q_{i+1})}{\log(b_1\cdots b_{i+1})}\right\}
\end{equation}
and $q'_{i}=1, k_j+2\le i\le k_{j+1}-1, q'_{k_{j+1}}=q_{k_{j+1}}, q'_{k_{j+1}+1}=q_{k_{j+1}+1}, j\ge 1.$

It follows from \eqref{tps} that
\[
\text{$\frac{\log (q_1'\cdot \cdots q_{k_{j}}')}{\log(b_1\cdots b_{k_j})}\le t<\frac{\log (q_1'\cdot \cdots q_{k_{j}}'q_{k_j+1})}{\log(b_1\cdots b_{k_j+1})}$ for $j\ge 1$.}
\]
So, due to assumption \eqref{conn}, we have
\[
\lim\limits_{j\to\infty}\frac{\log (q_1'\cdot \cdots q_{k_{j}}')}{\log(b_1\cdots b_{k_j})}=\lim\limits_{j\to\infty}\frac{\log (q_1'\cdot \cdots q_{k_{j}}'q_{k_j+1})}{\log(b_1\cdots b_{k_j+1})}=t.
 \]
On the other hand, for $j\ge 2$ and any $k_j<k<k_{j+1}$,  it follows from the definition of $\{k_j\}_{j=1}^\infty$ that $\frac{\log (q_1'\cdot \cdots q_{k}')}{\log(b_1\cdots b_{k})}\le t$. Hence, \eqref{ssp} holds.
\end{proof}

\begin{prop}
	For any $t\in(0,\ue \mu)$, there exists a spectrum $\Lambda_t$ of $\mu$ such that
\[
\dim \Lambda_t=t.
\]
\end{prop}

\begin{proof} Let $0<t<\ue \mu$ and $\{q_k'\}_{k=1}^\infty$ be the sequence in Lemma \ref{mid}. Write
\[
\Lambda_t^1=\bigcup_{k=1}^\infty \left\{\sum_{i=1}^k x_i(r_ib_1\cdots b_{i-1}):~\text{$x_i\in \Sigma_{q_i'}$ for $1\le i\le k$}\right\}.
\]

Then $\Lambda_t^1$ is an integer Moran set associated with the data $\{b_k\}, \{q_k'\}$and $\{r_k\}$. By Proposition \ref{Bdims}, Lemma \ref{mid} and the assumption \eqref{conn}, we have
 \[
 \begin{split}
 \dim \Lambda_t^1&=\limsup\limits_{k\to \infty}\frac{\log (q'_1\cdots q'_k)}{\log(q_k'r_kb_1\cdots b_{k-1})}\\
 &=\limsup\limits_{k\to \infty}\frac{\log (q'_1\cdots q'_k)}{\log(b_1\cdots b_{k})+\log (q_k'/q_k)}\\
 &=\limsup\limits_{k\to \infty}\frac{\log (q'_1\cdots q'_k)}{\log(b_1\cdots b_{k})}\quad \left(\frac{1}{M}\le \frac{q_k'}{q_k}\le 1\right)\\
 &=t
 \end{split}
 \]

For $k\ge 1$, let $\mathcal{D}'=\{q_n'\}_{n=1}^\infty$ and write
\[
\widetilde{\Sigma}^k_{\mathcal{D}'}:=\{(\sigma_1'\sigma_2'\ldots \sigma_k')\in \Sigma^k_{\mathcal{D}'}: \sigma_k'\not=0\}.
\]
Define $\Gamma_t=\cup_{k=1}^\infty \Gamma_k$, where $\Gamma_k=\{n=\sum_{j=1}^{k}\sigma_jq_1\cdots q_{j-1}:\sigma\in \widetilde{\Sigma}^k_{\mathcal{D}'}\}$.

Let $s_n=n$ for  $n\notin \Gamma_t$, and $s_n=0$ for $n\in \Gamma_t$. Let $\Lambda_t$ be the corresponding spectrum of $\mu$ with respect to the sequence $\{s_n\}_{n=1}^\infty$. It is easy to check that  $\Lambda_t^1=\{\lambda_n\}_{n\in \Gamma_t}$.  Write
\begin{equation}\label{deco}
\Lambda_t=\{\lambda_n\}_{n\in \Gamma_t}\cup\{\lambda_n\}_{n\notin \Gamma_t}:=\Lambda_t^1\cup\Lambda_t^2.
\end{equation}
%
%

We next use Lemma \ref{lac} to show that $\dim \Lambda_t^2=0$. In fact, as in Case 1, we have  $\dim(\Lambda(\{n\}))=0$. On the other hand, it is easy to check that $\Lambda_t^2\subset \Lambda(\{n\})$. By Lemma  \ref{inc} (i), we have $\dim\Lambda_t^2=0$.

It follows from \eqref{deco} and the finite stability of Beurling dimension that
\[
\dim \Lambda_t=\max(\dim \Lambda_t^1, \dim \Lambda_t^2) =t,
\]
 and therefore $\Lambda_t$ is the desired spectrum.
\end{proof}

The proof of Theorem  \ref{MR1} is completed.
\end{proof}


Finally, we present the proof of Theorem \ref{MR2}.

\begin{lemma}\label{lemca}
For $t\in[0,\ue \mu)$, the above $\Gamma_t$ and $\Gamma_t^c$ are all countably infinite.
\end{lemma}
\begin{proof}
By our choice as above, it is easy to see that they are countably infinite.	
\end{proof}
Given $t\in[0,\ue \mu)$.  For any $n\notin\Gamma_t$, we let $s_n=n^2$ or $n^2+1$. Then it is easy to see that $s_n<s_{n+1}$ and the Beurling dimensions of the corresponding sets $\Lambda_t^2$ are zeros. By choosing $s_n,n\notin \Gamma_t$ randomly from the above two choices and Lemma \ref{lemca}, we obtain that
 the level set
\[
L_t:=\{\Lambda:  \text{$\Lambda$ is a spectrum of $\mu$ and $\dim \Lambda =t$} \}
\]
has the cardinality of the continuum.

When $t=\ue \mu$. We prove it by some concrete constructions. Write
\[
\{-1,1\}^{\infty}=\left\{i_1i_2\cdots:~\text{$i_j\in\{-1,1\}$ for all $j\ge 1$}\right\}
\]
and
\[
1^\infty=11\cdots,~~(-1)^\infty=(-1)(-1)\cdots.
\]
We claim that for each $w=w_1w_2\cdots\in \{-1,1\}^{\infty}$, the set
\[
\Lambda^w=\bigcup_{k=1}^\infty \left\{\sum_{i=1}^k x_iw_i(r_ib_1\cdots b_{i-1}):~\text{$x_i\in \Sigma_{q_i}$ for $1\le i\le k$}\right\}
\]
is a spectrum of $\mu$. In fact, denote $\mu_n=\delta_{b_1^{-1}\mathcal{D}_1}\ast \delta_{(b_1b_2)^{-1}\mathcal{D}_2}\ast\cdots \ast \delta_{(b_1b_2\cdots b_n)^{-1}\mathcal{D}_n}$. We let
   $$T_n=\Big\{(b_1b_2\cdots b_n)^{-1}\sum_{i=1}^nx_iw_i(r_ib_1b_2\cdots b_{i-1}):~\text{$x_i\in \Sigma_{q_i}$ for $1\le i\le n$}\Big\}.$$
For any $n\in\N$, since \[
\begin{split}
(b_1b_2\cdots b_n)^{-1}\sum_{i=1}^{n}&(q_i-1)r_i b_1b_2\cdots b_{i-1}\\
&=\frac{(q_n-1)r_n}{b_n}+\frac{(q_{n-1}-1)r_{n-1}}{b_{n-1}b_n}+\cdots+\frac{(q_1-1)r_1}{b_1b_2\cdots b_n}\\
            &\leq \Big(1-\frac{1}{b_n}\Big)+\Big(\frac{1}{b_n}-\frac{1}{b_{n-1}b_{n}}\Big)+\cdots+\Big(\frac{1}{b_2\cdots b_n}-\frac{1}{b_1b_2\cdots b_n}\Big)\\
            &=1-\frac{1}{b_1b_2\cdots b_n}\\
            &<1,
\end{split}
\]
it follows that $T_n\subset[-1,1]$. Hence the set
\[
\mathcal{Z}(\widehat{\mu}_n):=\{\xi:\widehat{\mu}_n(\xi)=0\}=\bigcup_{k=0}^{n-1}b_1b_2\ldots b_kr_{k+1}(\Z\setminus q_{k+1}\Z)
\]
is separated from the set $T_n$ by $(r_1-1),$ uniformly in $n$. On the other hand, for any $w_k\in\{-1,1\}$ with $k\geq1$, it is easy to show that $(b_k^{-1}\mathcal{D}_k,w_kr_k\Sigma_{q_k})$ is a compatible pair, i.e., the matrix
   $$\frac{1}{\sqrt{q_k}}\left(e^{-2\pi i\frac{d_kw_kr_k l_k}{b_k}}\right)_{d_k\in \mathcal{D}_k,l_k\in\Sigma_{q_k}}$$
   is unitary.
According to Theorem 2.8 in \cite{Str00}, we have that $\Lambda^w$ is a spectrum of $\mu.$

By the proof of Proposition \ref{Bdims} and the symmetry relation between $\Lambda^{1^\infty}$ and $\Lambda^{(-1)^\infty}$, we have
\[
\dim(\Lambda^{(-1)^\infty})=\dim(\Lambda^{1^\infty})=\ue \mu\left(=\limsup\limits_{n\to \infty}\frac{\log q_1q_2\cdots q_n}{\log b_1b_2\cdots b_n}\right).
\]
Noting that the set $\{-1,1\}^\infty$ has the cardinality of the continuum, we only need to prove that $\dim(\Lambda^w)=\ue \mu$ for each $w\in\{-1,1\}^\infty$.

 Given  $w\in\{-1,1\}^\infty$. We first prove that $\dim(\Lambda^w)\geq\ue \mu$.   Denote
 \[
 \Lambda_k^w= \left\{\sum_{i=1}^k x_iw_i(r_ib_1\cdots b_{i-1}):~\text{$x_i\in \Sigma_{q_i}$ for $1\le i\le k$}\right\},
 \]
   and then $\Lambda^w=\bigcup_{k=1}^{\infty}\Lambda_k^w.$ Take $x=0$. For $k\ge 1,$ let $h_k=\sum_{i=1}^k (q_i-1)r_i(b_1\cdots b_{i-1})$. Then $h_k\leq b_1b_2\cdots b_{k}$. Moreover, note that $\Lambda^w_k\subset B(0,h_k)$ and $\Lambda^w_{k+1}\setminus \Lambda^w_k\subset \R\setminus B(0,h_k)$. In fact, the first inclusion relation is clear and the second inclusion relation holds because $r_{k+1}b_1b_2\cdots b_k\geq 2\sum_{i=1}^{k}(q_i-1)r_ib_1\cdots b_{i-1}$ (since $r_{k+1}\geq 2$). Then $\#(\Lambda^w\cap B(0,h_k))=q_1\cdots q_k$.
   Hence, by Proposition \ref{ec} we have
   \[
   \begin{split}
   	\dim \Lambda^w&\ge \limsup\limits_{k\to \infty}\frac{\log \#(\Lambda^w \cap B(0,h_k))}{\log h_k}\\
   	&\ge \limsup\limits_{k\to \infty}\frac{\log q_1q_2\cdots q_k}{\log b_1b_2\cdots b_k}=\ue \mu.
   \end{split}
   \]

We next prove that
\begin{equation}\label{coni}
\dim \Lambda^w\le \ue \mu.
\end{equation}
Let $h>0$. Let $x\in\R$. Then we can choose sufficiently large $k$ such that $[-b_1b_2\cdots b_{k-1},b_1b_2\cdots b_{k-1}]\subset B(x,h)\subset [-b_1b_2\cdots b_k,b_1b_2\cdots b_k]$. Then
 \[
 \begin{split}
	\#(\Lambda^w\cap B(x,h))&=\#(\Lambda^w_k\cap B(x,h))\\
	&\leq \#\left(\Lambda^{(-1)^\infty}_k\cap B(x,h)\right)+\#\left(\Lambda^{1^\infty}_k\cap B(x,h)\right).
\end{split}
\]
On the other hand, for any $\varepsilon>0$, note that $\dim\Lambda^{1^\infty}=\dim\Lambda^{(-1)^\infty}=\ue \mu$, we have
\[
\#\left(\Lambda^{(-1)^\infty}_k\cap B(x,h)\right)\le h^{\ue \mu+\varepsilon}, ~~ \#\left(\Lambda^{1^\infty}_k\cap B(x,h)\right)\le h^{\ue \mu+\varepsilon}
\]
for large enough $h$. Hence,
\[
\#(\Lambda^w\cap B(x,h))\leq 2h^{\ue \mu+\varepsilon}.
\]
Finally, \eqref{coni} holds due to Proposition \ref{ec} again.


\subsection*{Acknowledgements}
The authors would like to thank Professor Xinggang He for drawing our attention to these problems and reading the manuscript. The authors would like to thank Professor Meng Wu for his extremely valuable and helpful suggestions, which improves and simplifies the original edition. The project was supported by the National Natural Science Foundations of China (12171107, 12271534, 11971109), Guangdong NSF (2022A1515011844), the Foundation of Guangzhou University (202201020207, RQ2020070).

\end{document}